\documentclass[12pt]{amsart}
\usepackage{amssymb,latexsym}
\usepackage{amsfonts}
\usepackage{amsmath}
\usepackage[colorlinks,linkcolor=blue,anchorcolor=blue,citecolor=blue]{hyperref}
\usepackage{algorithm}
\usepackage{enumerate}
\usepackage{algpseudocode}
\usepackage{verbatim}
\usepackage{graphicx}
\usepackage{todonotes}

\newcommand\C{{\mathbb C}}
\newcommand\F{{\mathbb F}}
\newcommand\Q{{\mathbb Q}}
\newcommand\R{{\mathbb R}}
\newcommand\Z{{\mathbb Z}}

\newcommand\N{{\mathbb N}}

\newtheorem{theorem}{Theorem}
\newtheorem{lemma}[theorem]{Lemma}

\theoremstyle{remark}

\newcommand\rC{{\mathrm C}}
\newcommand\rH{{\mathrm H}}
\newcommand\rM{{\mathrm M}}
\newcommand\rL{{\mathrm L}}

\begin{document}

\title[Counting decomposable polynomials]{Counting decomposable polynomials with integer coefficients}

\author{Art\= uras Dubickas}
\address{Institute of Mathematics, Faculty of Mathematics and Informatics, Vilnius University, 
Naugarduko 24, Vilnius LT-03225,  Lithuania}
\email{arturas.dubickas@mif.vu.lt}

\author{Min Sha}
\address{School of Mathematical Sciences, South China Normal University, Guangzhou, 510631, China}
\email{min.sha@m.scnu.edu.cn}

\subjclass[2010]{11C08, 11R06, 12D05}

\keywords{Decomposable polynomial, height, degree, asymptotic formula}

\begin{abstract}
A polynomial over a ring is called decomposable if it is a composition of two nonlinear polynomials. 
In this paper, we obtain sharp lower and upper bounds for the number of decomposable polynomials with integer coefficients of fixed degree and bounded height. 
Moreover, we obtain asymptotic formulas for the number of decomposable monic polynomials of even degree. 
For example, the number of monic sextic integer polynomials which are decomposable and of height at most $H$ 
is asymptotic to $(16\zeta(3)-5/4)H^3$
as $H \to \infty$.  
\end{abstract}

\maketitle

\section{Introduction}

\subsection{Motivation and background}

Let $R$ be a ring. 
A univariate polynomial $f \in R[x]$ is called \textit{decomposable} (over $R$) if it is a composition of two nonlinear polynomials $g, h \in R[x]$, namely, $f=g \circ h$ which means that $f(x)=g(h(x))$;   
otherwise, it is called \textit{indecomposable}. 

Starting from the foundational work of Ritt \cite{Ritt1} in 1922 on polynomial decompositions over the complex numbers $\C$, 
polynomial decompositions have been extensively studied. See, for instance, 
 \cite{Barton, Blank, Gath90, Gath90b, Kozen} for algorithmic questions, 
\cite{Dorey, Eng, FM, Gath14, Levi, Sch, Tor, Zan} for structural properties, 
 \cite{BCD} for a so-called specialization problem, 
 and \cite{BDN} for a so-called spectrum problem. 
 
 The motivation also partly comes from \textit{van der Waerden's conjecture} \cite{Waerden}. 
 This conjecture asserts that for monic polynomials of degree $n$ with integral coefficients in the interval $[-H, H]$, 
 at most $O(H^{n-1})$ such polynomials have Galois group not equal to the full symmetric group $S_n$. 
 For $n \le 4$, van der Waerden's conjecture was proven by Chow and Dietmann \cite{CD}. 
 Most recently, Bhargava released a proof of van der Waerden's conjecture for all degrees $n$ \cite{Bhargava}. 
 The main sources of polynomials with Galois group strictly smaller than $S_n$ are reducible polynomials, decomposable polynomials, 
 and (possibly) polynomials with square discriminant. 

In this paper we will study a counting problem of decomposable polynomials. 
One can expect that the decomposable polynomials form a small minority among all polynomials over a field. (The same happens for polynomials with various restrictions which we previously considered in \cite{du1, du2, du3, du4}.)

Counting decomposable polynomials over finite fields was first considered by Giesbrecht \cite{Gie}. 
He showed that the decomposable polynomials form an exponentially small fraction of all polynomials. Recently,
von zur Gathen \cite{Gath15} gave general approximations to the 
number of decomposable polynomials with rapidly decreasing relative error bounds; 
see \cite{Gath13} for some earlier work and \cite{Blank13,GGZ,Zie} for further improvements. Moreover,
von zur Gathen also established results on counting decomposable multivariate polynomials over finite fields \cite{Gath11} 
and estimated the density of real and complex decomposable polynomials \cite{Gath17}. 

In this paper, we aim to estimate the number of decomposable polynomials with integer coefficients 
(that is, defined over the rational integers $\Z$) of fixed degree and bounded height. 
Here, the \textit{height} of a complex polynomial is the maximum of the absolute values of its coefficients. 

The counting model we choose in this paper is to fix the degree and vary the coefficients of polynomials, 
which is called \textit{the large box model} in \cite{BSK}. 
We remark that there is another model which is of recent interest: \textit{the bounded box model}. 
In this new model one chooses the coefficients uniformly from a finite set and studies 
the statistics when the degree  goes to the infinity. 
Recent works on this include papers by Bary-Soroker and Kozma \cite{BSK}, Bary-Soroker, Koukoulopoulos and Kozma \cite{BSKK}, 
Breuillard and Varj{\' u} \cite{BV}, and the thesis of Akkerman \cite{Yonathan} which 
considers statistics of decomposable polynomials for the first time in this model.

\subsection{Set-up and approach}  \label{sec:setup}
A polynomial $f \in \Z[x]$ is called \textit{origin-preserving} if $f(0) = 0$ 
(note that in \cite{Gath13,Gath15} it is called \textit{original}).  
For any polynomials $f, g, h \in \Z[x]$, it is easy to see that for any integers $a \ne 0$ and $b$,  
$$
f = g \circ h \quad \textrm{if and only if} \quad af + b = (ag + b)\circ h. 
$$
So, we only need to consider the decompositions of origin-preserving polynomials. 

Moreover, if $f$ is origin-preserving and $f = g \circ h$, then $f= g^* \circ h^*$, where $g^*(x) = g(x+h(0))$ and $h^*(x) = h(x) - h(0)$ 
satisfying  $h^*(0) = 0$ and $g^*(0) = 0$. 
Hence, for an origin-preserving polynomial, we only need to consider its decompositions by origin-preserving polynomials. 

Throughout the paper, let $n \geq 2$ and $H \geq 2$ be two positive integers.  
We denote by $D_n(H)$ the set of decomposable monic origin-preserving polynomials $f \in \Z[x]$ of degree $n$ and height at most $H$. 
Then, by the above discussion, the number of decomposable monic polynomials $f \in \Z[x]$ of degree $n$ and height at most $H$ is equal to 
$$
(2H + 1) \cdot \# D_n(H). 
$$

Recall that the \textit{content} of a non-constant polynomial in $\Z[x]$ is the greatest common divisor of its coefficients. 

For non-monic polynomials in $\Z[x]$, by the above, we only need to consider non-monic origin-preserving polynomials with positive leading coefficient 
and moreover (by Lemma~\ref{lem:content}) only polynomials with content 1. 
Let $D_n^*(H)$ be the set of decomposable origin-preserving polynomials in $\Z[x]$ (not necessarily monic) 
 of degree $n$ with positive leading coefficient, content $1$ and height at most $H$. 
 
 Clearly, $D_n(H)$ is a subset of $D_n^*(H)$.  
 Our aim is to estimate the cardinalities $\#D_n(H)$ and $\#D_n^*(H)$ with respect to $H$. 
 
 By definition, if a polynomial $f$ is decomposable, then its degree $n=\deg f$ is a composite number. 
So, if $n$ is a prime, we clearly have $$\#D_n(H) =0 \quad \text{and} \quad \#D_n^*(H) =0.$$ 
Hence, it suffices to consider the case when $n$ is a composite number. 
 
We emphasize that in this paper we fix the degree $n$. Then, the key point is to control the heights of polynomials.
So the approach here is totally different from those in the finite field case (such as in \cite{Gath13, Gath15}). 
There are two reasons why we fix $n$: one is for simplicity, and the other is that we care more about the height (as in \cite{du1, du2, du3, du4}). 

Roughly speaking, our approach to count decomposable polynomials $f = g \circ h$  can be separated into two steps. 
The first step is to bound the coefficients of $g$ and $h$ in terms of the height of $f$ (which is at most $H$). A key result in this direction is Lemma~\ref{lem:height}. 
The second step is to count the number of possible candidates for $g$ and $h$ in order to deduce the number of decomposable polynomials.  

In presentation of our results we will use the Landau symbol $O$ and the Vinogradov symbol $\ll$. 
Recall that the assertions $U=O(V)$ and $U \ll V$ are both equivalent to the inequality $|U|\le cV$ with some constant $c>0$. 
Besides, $U \asymp V$ means that $U \ll V \ll U$. 
In this paper, if the constants implied in the symbols $O, \ll, \asymp$ depend on the degree $n$, 
we will write them respectively as $O_n, \ll_n, \asymp_n$.

\section{Main results}

\subsection{The monic case}

We begin with the following result which describes the size of $D_n(H)$:

 \begin{theorem}  \label{thm:monic}
 	Let $\ell$ be the smallest prime divisor of a composite number $n \geq 6$. Then, 
 	$$\#D_n(H) \asymp_n H^{\frac{n}{\ell}-1}.$$
 \end{theorem}

Here, we compare Theorem \ref{thm:monic} with related results over finite fields of von zur Gathen \cite{Gath15}. 
Let $n,\ell$ be as above. From \cite[Main Theorem (ii)]{Gath15}, we know that the number of decomposable monic origin-preserving
polynomials of degree $n$ over a finite field $\F_p$ (where $p$ is a prime number) is $\asymp p^{\ell + n/\ell -2}$. 
Since $\ell+n/\ell-2 \ge n/\ell $, the growth rate is  larger than that in Theorem \ref{thm:monic} for $H=p$. 
Thus, these two settings behave quite differently. 
This also suggests that there are much more indecomposable polynomials over $\Z$ which are decomposable modulo $p$ 
than  decomposable polynomials over $\Z$. 

In addition, Widmer \cite[Theorem 1.1]{Widmer} 
(see the discussions there, or see \cite[Equation (5)]{Bhargava}) 
gave an essentially optimal upper bound $O(H^{\frac{n}{2} +2})$ for 
the number of irreducible monic polynomials $f \in \Z[x]$ of degree $n$ and height at most $H$ 
and corresponding to number fields $\Q[x]/(f)$ having a non-trivial subfield. 
Clearly, typical examples of this kind of polynomials are irreducible decomposable polynomials. 
Note that Theorem~\ref{thm:monic} implies that the number of monic decomposable polynomials in $\Z[x]$ 
of degree $n$ and height at most $H$ is $O_n (H^{\frac{n}{\ell}})$. 
So, these two estimates suggest that irreducible decomposable polynomials occupy a small proportion
of those polynomials essentially counted in \cite[Theorem 1.1]{Widmer}. 

When $n$ is even, namely, $\ell=2$ in Theorem~\ref{thm:monic}, we obtain more explicit asymptotic results for $D_n(H)$.
Firstly, we give an asymptotic formula for $n=4$ (the case $n=4$ is missing in Theorem~\ref{thm:monic}):

\begin{theorem}  \label{thm:monicasy0}
	For $n=4$ we have 
	\begin{equation*}\label{pi1}
		\#D_4(H)=2H\log H+O(H).
	\end{equation*}
\end{theorem}

For $n=6$ we will prove the following explicit asymptotic formula.  
(Throughout, $\zeta(\cdot)$ is the Riemann zeta function.)

\begin{theorem}  \label{thm:monicasy00}
	For $n=6$ we have
	\begin{equation*}\label{pi3}
		\#D_6(H)= \Big(8\zeta(3)-\frac{5}{8}\Big)H^{2}+O(H^{\frac{3}{2}}). 
	\end{equation*}
\end{theorem}

Finally, for other even $n \geq 8$ the following asymptotic formula is true:

\begin{theorem}  \label{thm:monicasy}
	For each even $n \geq 8$ there is a constant $\kappa_n>0$ such that 
	\begin{equation*}\label{pi2}
		\#D_n(H)=\kappa_n H^{\frac{n}{2}-1} + O_n(H^{\frac{n}{2}-2}).
	\end{equation*}    
\end{theorem}

We remark that the constant $\kappa_n$ for even $n \geq 8$ is 
of the form $2^{\frac{n}{2}} \zeta\big(\frac{n(n-2)}{8}\big)+r_n$ with
some rational number $r_n$ (see \eqref{eq:kappa}). 

It seems very likely that Theorem~\ref{thm:monicasy} holds for all composite $n \geq 8$ with exponent $n/\ell-1$, where $\ell$ is the smallest prime divisor of $n$. This, if proved, combined with Theorem~\ref{thm:monicasy00} would give a full asymptotic version of Theorem~\ref{thm:monic}
for each composite number $n \geq 6$.  

In order to prove Theorem~\ref{thm:monic}, we need to make some more detailed studies. 
For any composite number $n$ with a proper divisor $d\ge 2$, let $D_{n,d}(H)$ be the set of polynomials $f\in D_n(H)$  
such that $f=g \circ h$ for some monic origin-preserving polynomials $g,h\in \Z[x]$ with $\deg g = d$ and $\deg h =n/d$. 
Clearly, 
\begin{equation}  \label{eq:Dn}
D_n(H) = \bigcup_{d \mid n, \, 2 \leq d <n} D_{n,d}(H).
\end{equation}

With this notation we will prove the following:

\begin{theorem}  \label{thm:monic-nd}
	Let $n \geq 4$ be a composite number with a proper divisor $d \ge 2$. 
	Write $n = dm$. 
	Then, we have 
	\begin{equation}\label{gui1}
	\#D_{n,d}(H) \asymp_n H^{d-1}
	\end{equation}
	 if $d(d-1) > 2(m-1)$,
		\begin{equation}\label{gui2}
		\#D_{n,d}(H) \asymp_n H^{d-1}\log H
			\end{equation}
		  if $d(d-1) = 2(m-1)$,
		and 
		\begin{equation}\label{gui3}
		\#D_{n,d}(H) \asymp_n H^{\frac{d-1}{2}+\frac{m-1}{d}}
			\end{equation}
		 if $d(d-1) < 2(m-1)$.
\end{theorem}

Note that for $d \ge m \ge 2$ we have either $d(d-1) > 2(m-1)$ or $d=m=2$.  In the case when $2 \leq d < m$, we have $$\frac{d-1}{2} + \frac{m-1}{d} < m-1.$$ 
Thus, combining \eqref{eq:Dn} with Theorem \ref{thm:monic-nd}, we 
immediately deduce Theorem~\ref{thm:monic}.

Actually, Theorem \ref{thm:monic-nd} tells us more information. 
By Theorem \ref{thm:monic-nd}, we see that the main contribution to 
$D_n(H)$ comes from the set $D_{n,n/\ell}(H)$. 

We denote by $I_n(H)$  the set of polynomials $f\in D_n(H)$ 
such that $f=g \circ h$ for some \textit{indecomposable} monic origin-preserving polynomials $g,h\in \Z[x]$ with $\deg g =n/\ell$ and $\deg h =\ell$, 
where $\ell$ is the smallest prime divisor of $n$. 
In particular, if $n=\ell^2$, we have $I_n(H) = D_{n,\ell}(H)=D_n(H)$, where one should note that $\ell$ is a prime.  

\begin{theorem}  \label{thm:monic-R}		
Let $\ell$ be the smallest prime divisor of a composite number $n$. Then,  we have 
\begin{equation*}   
\#D_n(H) = \#I_n(H) + 
\left\{
\begin{array}{rl}
O(H^{\frac{3}{2}}) & \textrm{if $n=6$}, \\
O_n(H^{\frac{n}{\ell}-2})  & \text{otherwise}. 
\end{array} \right.
\end{equation*}
\end{theorem}

From Theorems \ref{thm:monic} and \ref{thm:monic-R}, we know that 
 the main contribution to $D_n(H)$  comes from the set $I_n(H)$.

\subsection{The non-monic case}

The following result describes the size of $D_n^*(H)$:

\begin{theorem}  \label{thm:non-monic}
	Let $\ell$ be the smallest prime factor of a composite number $n \geq 4$. Then, 
	\begin{equation*}
		\#D_n^*(H) \asymp_n  H^{\frac{n}{\ell}}. 
	\end{equation*}
\end{theorem}

Note that the exponent for $H$ in Theorem~\ref{thm:non-monic}
is greater by one than that in Theorem~\ref{thm:monic}. This is a natural phenomenon  and occurs in similar situations (monic vs. non-monic) in 
\cite{du1, du2, du3, du4} and also in \cite{du10,du11}. 

As in the monic case, for any composite number $n$ with  a proper divisor $d  \ge 2$, 
let $D_{n,d}^*(H)$ be the set of polynomials $f\in D_n^*(H)$ of degree $n$ 
such that $f=g \circ h$ for some origin-preserving polynomials $g,h\in \Z[x]$ with positive leading coefficient and with content 1 
and also $\deg g =d$ and $\deg h =n/d$. 
Clearly, 
\begin{equation}  \label{eq:Dn*}
D_n^*(H) = \bigcup_{d \mid n, \, 2 \leq d <n} D_{n,d}^*(H). 
\end{equation}

Here is an analogue of Theorem~\ref{thm:monic-nd} in the non-monic case. 

\begin{theorem}  \label{thm:non-monic-nd}
		Let $n \geq 4$ be a composite number with a proper divisor $d \ge 2$. 
	Write $n = dm$. 
	Then, we have 
	\begin{equation}\label{gui11}
		\#D_{n,d}^*(H) \asymp_n H^{d}
	\end{equation}
	if $d(d+1) > 2m$,
	\begin{equation}\label{gui21}
		\#D_{n,d}^*(H) \asymp_n H^{d}\log H
	\end{equation}
	if $d(d+1) = 2m$,
	and 
	\begin{equation}\label{gui31}
		\#D_{n,d}^*(H) \asymp_n H^{\frac{d-1}{2}+\frac{m}{d}}
	\end{equation}
	if $d(d+1) < 2m$.
\end{theorem}

Notice that for $d \ge m \ge 2$  we certainly have  $d(d+1) > 2m$.  
When $2 \le d < m$, we have $$\frac{d-1}{2} + \frac{m}{d} < m.$$ 
Thus, combining \eqref{eq:Dn*} with Theorem \ref{thm:non-monic-nd}, we derive Theorem~\ref{thm:non-monic}.

Similarly, Theorem \ref{thm:non-monic-nd} also tells us more information. 
By Theorem \ref{thm:non-monic-nd}, the main contribution to 
$D_n^*(H)$ comes from the set $D_{n,n/\ell}^*(H)$.
Now, we denote by $I_n^*(H)$  the set of polynomials $f\in D_n^*(H)$ 
such that $f=g \circ h$ for some \textit{indecomposable} origin-preserving polynomials $g,h\in \Z[x]$ with positive leading coefficient and with content 1 
and also $\deg g =n/\ell$ and $\deg h =\ell$. 
In particular, if $n=\ell^2$, we have $I_n^*(H) = D_{n, \ell}^*(H) = D_n^*(H)$. 

\begin{theorem}  \label{thm:non-monic-R}
Let $\ell$ be the smallest prime factor of a composite number $n$. 
Then, we have 
\begin{equation*}   
\#D_n^*(H) = \#I_n^*(H) + 
\left\{
\begin{array}{rl}
O(H^2 \log H) & \textrm{if $n=6$}, \\
O_n(H^{\frac{n}{\ell} - \frac{3}{2}})  & \text{otherwise}. 
\end{array} \right.
\end{equation*}
\end{theorem}

According to Theorems \ref{thm:non-monic} and \ref{thm:non-monic-R}, we see that 
 the main contribution to $D_n^*(H)$ comes from $I_n^*(H)$.
 
In all what follows we first gather some basic auxiliary results in Section~\ref{sec:pre} 
and then prove the main results in Sections~\ref{sec:monic} and \ref{sec:non-monic}.

\section{Preliminaries} 
\label{sec:pre}

\subsection{Some basic results}
For a complex polynomial of degree $n$ 
$$
f(x)=a_nx^n + \cdots + a_1x + a_0 = a_n \prod_{i=1}^{n}(x-\alpha_i) \in \C[x],
$$
 its \textit{height} is defined by 
$$
\rH(f) = \max \{|a_0|,|a_1|, \ldots, |a_n|\}, 
$$
and its \textit{Mahler measure} by 
$$
\rM(f) = |a_n| \prod_{i=1}^{n} \max\{1, |\alpha_i|\}. 
$$
For each $f \in \C[x]$ of degree $n$, these quantities are related by the following well-known inequality
\begin{equation}\label{eq:Mahler}
\rH(f) 2^{-n} \leq \rM(f) \leq \rH(f) \sqrt{n+1}; 
\end{equation}
for instance, see \cite[(3.12)]{Waldschmidt2000}.   
In addition, the \textit{length} of $f$ is defined by   
$$
\rL(f) = |a_0|+ |a_1|+ \cdots + |a_n|.  
$$
Clearly, we have $$\rH(f) \le \rL(f).$$ 
Besides, for two non-zero polynomials $g,h\in \C[X]$ with $g = c_d x^d + \cdots + c_1x + c_0$, it is easy to see that 
\begin{equation}  \label{eq:length}
\rL(g \circ h) \le |c_d| \rL(h)^d + \cdots + |c_1| \rL(h) + |c_0|.
\end{equation}

Next, if $f \in \Z[x]$ is non-constant, as usual, we define the \textit{content} of $f$, denoted by $\rC(f)$, to be 
the greatest common divisor of its coefficients, that is, 
$$
\rC(f) = \gcd(a_0, a_1, \ldots, a_n). 
$$

Essentially, the following lemma is a relation between the heights of polynomials in a polynomial decomposition $f=g\circ h$. 

\begin{lemma} \label{lem:height}
	Let $f = g \circ h \in \C[x]$ be of degree $n$, where 
	$$g(x)=a_dx^d+\dots+a_1x+a_0 \in \C[x], \>\> a_d \ne 0,$$ is of degree $d$ and  
	$h(x) \in \C[x]$ is of degree $m=n/d$ satisfying $h(0)=0$. 
	Then, 
	\begin{equation}\label{viens}
	|a_j| \rH(h)^j \leq 2^n {d \choose j} \rM(f)
	\end{equation}
	for each $j=0,1,\dots,d$. 
	
	In particular, for $a_d = 1$ and $d=n/2 \geq 2$ we have
	\begin{equation}\label{viens00}
		\rH(h) < 6\rH(f)^{\frac{1}{d}}. 
	\end{equation}
\end{lemma}

\begin{proof}
	Assume that
	$$
	g(x) = a_d \prod_{i=1}^{d} (x-\alpha_i).
	$$
	Then, 
	$$
	f(x) = g \circ h(x) =a_d \prod_{i=1}^{d} (h(x)-\alpha_i),
	$$
	and hence, by the multiplicativity of Mahler's measure,
	$$\rM(f)=|a_d| \prod_{i=1}^d \rM(h(x)-\alpha_i).$$ Combining this with
	$$\rM(h(x)-\alpha_i) \geq \frac{\rH(h(x)-\alpha_i)}{2^m} = \frac{\max\{\rH(h),|\alpha_i|\}}{2^m},$$
	where the inequality holds by \eqref{eq:Mahler} and the equality by $h(0)=0$,
	we derive that 
	$$|a_d|\prod_{i=1}^d \max\{\rH(h),|\alpha_i|\} \leq 2^{md}\rM(f)=2^n \rM(f).$$ 
	
	Therefore, to deduce \eqref{viens} it suffices to show that
	\begin{equation}\label{viens2}
	|a_j| \rH(h)^j \leq |a_d| {d \choose j} \prod_{i=1}^d \max\{\rH(h),|\alpha_i|\}
	\end{equation}
for $j=0,1,\dots,d$.

	 Without loss of generality we assume that $|\alpha_1| \geq |\alpha_2| \geq \dots \geq |\alpha_d|$. Then,
	$$
	\frac{|a_j|}{|a_d|}=\Big|\sum_{1 \leq i_1<\dots<i_{d-j} \leq d}\alpha_{i_1}\dots\alpha_{i_{d-j}}\Big| \leq {d \choose d-j} \prod_{i=1}^{d-j} |\alpha_i|={d \choose j} \prod_{i=1}^{d-j} |\alpha_i|,
	$$
	where the last product is empty for $j=d$. 
	Hence,
	$$
	\frac{|a_j|\rH(h)^j}{|a_d|} \leq {d \choose j} \rH(h)^j \prod_{i=1}^{d-j} |\alpha_i| \leq {d \choose j} \prod_{i=1}^d \max\{\rH(h),|\alpha_i|\},
	$$
	which implies \eqref{viens2} and then \eqref{viens}.
	
	Finally, \eqref{viens00} follows from \eqref{viens} with $j=d$ combined with
	the upper bound in \eqref{eq:Mahler} and the inequality $(2d+1)^{\frac{1}{2d}}<1.5$ for $d \geq 2$. 
	\end{proof}

We also need the next simple lemma. 

\begin{lemma}  \label{lem:factor}
Let $\ell$ be the smallest prime divisor of a composite number $n$ with $n \ne 6$.  
Then, for any divisor $d \ge 2$ of $n$ with $d \ne n/\ell$, we have $d \le n/\ell-2$. 
\end{lemma}

\begin{proof}
Since $\ell$ is the smallest prime divisor of $n$, we must have $d \le n/\ell$. 
For a contradiction we assume that $d = n/\ell-1$. 
Then, we see that 
$n/\ell$ and $n/\ell-1$ are two coprime factors of $n$, and so $(n/\ell)(n/\ell - 1) \le n$. 
This implies $n/\ell-1\le \ell$, and hence $n \le \ell(\ell + 1)$. 
Now, due to the choice of $\ell$ and $n \ge \ell^2$, we must have $n = \ell(\ell + 1)$. But then $\ell+1$ must be a prime number, so that 
$\ell=2, d=2, n=6$ is the only possibility, which contradicts with the assumption $n \ne 6$. 
\end{proof}

Since standard results stating that with positive probability $k$ random integers are relatively prime (see, e.g., \cite{nym}) 
are not sufficient for our purpose, we also record the following:

\begin{lemma}  \label{lem:factor11}
Given integers $k \geq 2$ and $X_1,\dots,X_k \geq 1$, there are at least
$0.355 X_1\cdots X_k$  collections of positive integers $a_1,\dots,a_k$ satisfying
$1 \leq a_i \leq X_i$ for $i=1,\dots,k$ and 
$$\gcd(a_1,\dots,a_k)=1.$$   
\end{lemma}

\begin{proof}
Suppose $a_1,\dots,a_k$ is a collection of positive integers satisfying $1 \leq a_i \leq X_i$ for $i=1,\dots,k$ and $\gcd(a_1,\dots,a_k)>1$.
Then, there is a prime number $p \leq \min_{1 \leq i \leq k} X_i$ such that $p|a_i$
for $i=1,\dots,k$. The number of such collections (corresponding to this $p$) is clearly at most $(X_1/p) \dots (X_k/p)=p^{-k}X_1\dots X_k$. The number of all such (unsuitable) collections is thus less than or equal to 
$$X_1\dots X_k\sum_{ p \leq \min_{i}X_i} p^{-k} < X_1\dots X_k \sum_{j=2}^{\infty} j^{-k} = (\zeta(k)-1) X_1\dots X_k.$$ Consequently, the number of suitable collections (those with relatively prime integers $a_1,\dots,a_k$ in the range $1 \leq a_i \leq X_i$ for $i=1,\dots,k$) is at least
$$( 2-\zeta(k)) X_1 \dots X_k \geq (2-\zeta(2))X_1\dots X_k > 0.355 X_1 \dots X_k,$$
as claimed.
\end{proof}

\subsection{Decomposability of polynomials} 

We start with a simple result about non-monic polynomials in $\Z[x]$. 

\begin{lemma}  \label{lem:content}
For any origin-preserving polynomial $f \in \Z[x]$, 
$f$ is decomposable if and only if $f / \rC(f)$ is decomposable. 
\end{lemma}

\begin{proof}
Assume $\rC(f) > 1$. 
The sufficiency part is trivial. 
We only need to prove the necessity part. 
For this, it suffices to show that if $f$ is decomposable, then for a prime divisor $\ell$ of $\rC(f)$, 
$f/\ell$ is decomposable. 

Assume $f= g \circ h$ in $\Z[x]$ with origin-preserving polynomials $g,h$. 
We claim that either $\ell \mid \rC(g)$ or $\ell \mid \rC(h)$. 
In fact, by contradiction suppose that $\ell \nmid \rC(g)\rC(h)$.
Then, let $i,j$ be the largest integers such that $\ell$ does not divide the coefficient $a_i$ for $x^i$ in $g$ and  
the coefficient $b_j$ for $x^j$ in $h$ respectively.  
Then, looking at the coefficient for $x^{ij}$ in $g \circ h = f$ which must be divisible by $\ell$ we get $\ell \mid a_i b_j^i$. This 
contradicts with $\ell \nmid a_i b_j$. 

Now, if $\ell \mid \rC(g)$, we have $f/\ell = (g/\ell) \circ h$. 
If $\ell \mid \rC(h)$, write $g = a_d x^d + \cdots + a_1 x$ and define $g^* = a_d \ell^{d-1} x^d + \cdots + a_2 \ell x^2 +  a_1 x$, 
then we have $f/\ell = g^* \circ (h/\ell)$. 
This completes the proof. 
\end{proof}

Let $F$ be a field of characteristic $p$. 
For any integer $d \ge 1$, define the set 
$$
P_d(F) = \{f \in F[x]:\, \deg f = d, \textrm{$f$ monic and origin-preserving} \}. 
$$

In general the decomposition into monic origin-preserving polynomials is not unique, 
but it can happen once the degrees of polynomials are fixed. 
For this, we reproduce the following well-known result; see \cite[Corollary 2.3 (ii)]{Gath90} (see, also, \cite[Fact 3.1(i)]{Gath13}). 

\begin{lemma}\label{lem:injective}
For any integers $d,m \ge 1$,  if $p$ does not divide $d$ (automatically including the case $p=0$), 
then the composition map from $P_d(F) \times P_m(F)$ to $P_{dm}(F)$, sending $(g,h)$ to $g \circ h$, is injective. 
\end{lemma}

The following lemma is an analogue of Lemma~\ref{lem:injective} about  non-monic polynomials in $\Z[x]$. 
It shows that we count the ``right" polynomials in the set $D_n^*(H)$. 

For any integer $d \ge 1$, define the set 
\begin{align*}
N_d = \{f \in \Z[x]:\, & \deg f = d, \rC(f)=1, \textrm{$f$ is origin-preserving}, \\
&  \textrm{$f$ has positive leading coefficient} \}. 
\end{align*}

\begin{lemma}  \label{lem:injective-Z}
For any integers $d,m \ge 1$,  
the composition map from $N_d \times N_m$ to $N_{dm}$, sending $(g,h)$ to $g \circ h$, is injective. 
\end{lemma}

\begin{proof}
First, it is easy to see that the composition map is well-defined. 

Now, suppose that there exist polynomials $g_1, g_2 \in N_d$ and $h_1, h_2 \in N_m$ such that $g_1 \circ h_1 = g_2 \circ h_2$. 
Write $f = g_1 \circ h_1$, and let $a$ be the leading coefficient of $f$. 
Then, $f^* = f/a$ is a monic origin-preserving polynomial in $\Q[x]$. 

Let $a_1, a_2$ be the leading coefficients of $h_1, h_2$ respectively. 
Define $g_1^* = a^{-1}g_1(a_1 x), h_1^* = a_1^{-1} h_1$ and $g_2^* = a^{-1}g_2(a_2x), h_2^* = a_2^{-1} h_2$. 
Then, we have 
$$
f^* = g_1^* \circ h_1^* = g_2^* \circ h_2^*.
$$
Since $f^*, h_1^*, h_2^*$ are monic origin-preserving polynomials in $\Q[x]$, so are the polynomials $g_1^*, g_2^*$. 
Noticing $\deg g_1^* = \deg g_2^* = d$, by Lemma~\ref{lem:injective}, we must have 
$$
g_1^* = g_2^*, \qquad  h_1^* = h_2^*.
$$
Then, we must have $a_2 h_1 = a_1 h_2$. This, combined with $\rC(h_1) = \rC(h_2) = 1$, implies $a_1 = a_2$, 
and so $h_1 = h_2$ which yields $g_1 = g_2$. 
This completes the proof. 
\end{proof}

\section{Proofs for the monic case} 
\label{sec:monic}

\subsection{Proof of Theorem \ref{thm:monic-nd}}

Recall that $n =dm$, where $d,m \geq 2$. 
Note that for any polynomial $f \in D_{n,d}(H)$, 
we have $f=g\circ h$ for some monic origin-preserving polynomials $g$ and $h$ with $\deg g=d$ and $\deg h=m$.  
By  \eqref{viens} in Lemma~\ref{lem:height} and \eqref{eq:Mahler},  we have 
\begin{equation}   \label{eq:Hh}
\rH(h) \le c_1 \rH(f)^{\frac{1}{d}} \le c_1 H^{\frac{1}{d}}, 
\end{equation}
where $c_1 > 0$ is a constant depending only on $n$. 
Set $b=\rH(h)$. Then, at least one coefficient of $h$ is equal to $\pm b$. 
For each $1\le b \le c_1 H^{1/d}$ (see \eqref{eq:Hh}), 
since $h$ is monic and origin-preserving,  $h$ has at most $2(2b+1)^{m-2}$ possibilities. 
In addition, by Lemma~\ref{lem:height}, there exists a constant $c_2 > 0$ depending only on $n$ such that 
the absolute value of the coefficient of each monomial $x^j$ in $g$, $1 \le j \le d-1$, is at most 
$$
c_2 \rH(f)/\rH(h)^{j} = c_2 \rH(f)/b^{j} \le c_2 H /b.
$$ 
By estimating the possibilities of $g$ and $h$, we deduce that 
\begin{equation*}
\begin{split}
\#D_{n,d}(H) & \le \sum_{1 \leq b \leq c_1 H^{1/d}} 2(2b+1)^{m-2} \cdot \frac{(c_2H)^{d-1}}{b^{1+2+\cdots + (d-1)}} \\
 & \ll_n \sum_{1 \leq b \leq c_1 H^{1/d}} b^{m-2}\cdot \frac{H^{d-1}}{b^{1+2+\cdots + (d-1)}} \\
 & \ll_n H^{d-1}\sum_{1 \leq b \leq c_1 H^{1/d}} \frac{1}{b^{\frac{d(d-1)}{2}+2-m}} \\ 
 & \ll_n \left\{
\begin{array}{rl}
H^{d-1} & \textrm{if $d(d-1) > 2(m-1)$}, \\
H^{d-1} \log H  & \text{if $d(d-1) = 2(m-1)$}, \\ 
 H^{\frac{d-1}{2}+\frac{m-1}{d}} & \text{if $d(d-1) < 2(m-1)$}. 
\end{array} \right.
 \end{split}
\end{equation*}  
This gives the upper bounds in all three formulas.

To get the lower bounds observe that
for polynomials of the form 
$$
(x^{d} + a_{d-1} x^{d-1} + \cdots + a_1 x ) \circ x^m  = x^{dm} + a_{d-1} x^{(d-1)m} + \cdots + a_1 x^m, 
$$
where $|a_i| \le H$ for each $1 \le i \le d-1$, we obtain the lower bound 
$$
	\#D_{n,d}(H)  \ge (2H)^{d-1}. 
$$
This completes the proof of \eqref{gui1}. 

To derive the lower bound in \eqref{gui3} we consider monic origin-preserving $h \in \Z[x]$ of degree $m$ satisfying $\rH(h) \leq H^{1/d}/(2m)$ 
and monic origin-preserving $g(x)=x^d+a_{d-1}x^{d-1}+\dots+a_1x \in \Z[x]$ with $|a_i| \leq H^{(d-i)/d}$. Then, as $\rH(f) \le \rL(f)  = \rL(g \circ h)$
 and $\rL(h) \le H^{1/d}/2$ (for $H$ large enough), using \eqref{eq:length} we obtain 
\begin{align*}
\rH(f)  \leq  \Big(\frac{H^{\frac{1}{d}}}{2}\Big)^d+\sum_{i=1}^{d-1} H^{\frac{d-i}{d}} \Big(\frac{H^{\frac{1}{d}}}{2}\Big)^i =H \sum_{i=1}^d \frac{1}{2^i}<H.
\end{align*}
By Lemma~\ref{lem:injective}, distinct pairs $(g,h)$ give distinct monic $f$ satisfying $\rH(f) < H$. For $H$ large enough there are at least  
$$
\Big(\frac{H^{\frac{1}{d}}}{2m}\Big)^{m-1} \prod_{i=1}^{d-1} H^{\frac{d-i}{d}} = \frac{1}{(2m)^{m-1}} H^{\frac{m-1}{d}+\frac{d-1}{2}}
$$ 
of such pairs, which proves the lower bound in \eqref{gui3}. 

Finally, to derive the lower bound in \eqref{gui2} under assumption 
$d(d-1)=2(m-1)$
we fix an integer $B$ satisfying 
\begin{equation}\label{gui4}
1 \leq B \leq H^{\frac{1}{2d}}
\end{equation}
 and consider monic origin-preserving $h$
of height exactly $B$. There are at least $B^{m-2}$ of them. To each of them we assign the monic origin-preserving polynomial $g(x)=x^d+a_{d-1}x^{d-1}+\dots+a_1x \in \Z[x]$ with $|a_i| \leq H/(n^dB^i)$.
Then, as $\rH(f) \le \rL(f)$  and  $\rL(h) \le mB$, using \eqref{eq:length} we obtain
$$
\rH(f) \leq  (mB)^d+\sum_{i=1}^{d-1} \frac{H}{n^dB^i} (mB)^i \leq m^d \sqrt{H}+\frac{(d-1)m^{d-1}H}{n^d}<H
$$
for $H$ large enough. As above, by Lemma~\ref{lem:injective}, distinct pairs $(g,h)$ give distinct monic origin-preserving $f$ satisfying $\rH(f) < H$. 
Since $$m-2-\frac{d(d-1)}{2}=-1,$$ for each sufficiently large $H$ there are at least  
$$B^{m-2} \prod_{i=1}^{d-1} \frac{H}{n^dB^i} = \frac{H^{d-1}}{n^{d(d-1)}B}$$ of such pairs. 
Summing over $B$ as in the range \eqref{gui4} gives the logarithmic factor, which yields the lower bound in \eqref{gui2}. This completes the proof of the theorem.

\subsection{Proof of Theorem \ref{thm:monic-R}}

For $n=6$, by Theorem \ref{thm:monic-nd} we have 
$$
\#D_{6,2}(H) \ll H^{\frac{3}{2}}. 
$$
Since $D_6(H) = D_{6,3}(H) \cup D_{6,2}(H)$ and $I_6(H)=D_{6,3}(H)$, we indeed have 
$$
\#D_6(H) = \#I_6(H) + O(H^{\frac{3}{2}}). 
$$

For $n = \ell^2$ it is clear that $D_n(H) = D_{n,\ell}(H) = I_n(H)$. 

In all what follows, we suppose that $n \ne 6$ and $n \ne \ell^2$. 
Then, due to the choice of $\ell$, we must have  $n/\ell \ge 4$. 
According to Theorem \ref{thm:monic-nd}, it suffices to consider three cases regarding the pairs $(d,m)$ of integers $d,m\ge 2$ with $n=dm$. 

The first case is $d(d-1) > 2(m-1)$. 
Given such a pair $(d,m)$, by Theorem~\ref{thm:monic-nd},  we get 
$$
\#D_{n,d}(H) \ll_n H^{d-1}. 
$$ 
By Lemma~\ref{lem:factor}, if $2 \le d \ne n/\ell$, we have $d \le n/\ell - 2$, and so
\begin{equation}  \label{eq:R1}
\#D_{n,d}(H) \ll_n H^{\frac{n}{\ell}-3}. 
\end{equation}

The second case is $d(d-1)=2(m-1)$. 
Clearly, $d < n/\ell$. 
Applying Theorem~\ref{thm:monic-nd} and Lemma~\ref{lem:factor}, we obtain 
\begin{equation}  \label{eq:R2}
\#D_{n,d}(H) \ll_n H^{\frac{n}{\ell}-3}\log H. 
\end{equation}

The third case is $d(d-1) < 2(m-1)$. 
By Theorem \ref{thm:monic-nd}, we have
$$
\#D_{n,d}(H)  \ll_n H^{\frac{d-1}{2}+\frac{m-1}{d}}. 
$$ 
If $d=2$, then $\ell=2, m=n/\ell \ge 4$. Hence, 
$$
\frac{d-1}{2}+\frac{m-1}{d} = \frac{m}{2}  = \frac{n}{2\ell}  \le \frac{n}{\ell} - 2, 
$$
which implies that 
\begin{equation}   \label{eq:R3-1}
\#D_{n,d}(H) \ll_n  H^{\frac{n}{\ell}-2}. 
\end{equation}  
Next, assume that $d \ge 3$.
Then, in view of $d(d-1) < 2(m-1)$ we have $d(d-1) \le 2(m-1) - 2 = 2(m-2)$ (because $d(d-1)$ is even). So, noticing $n/\ell \ge 4$, we obtain 
\begin{align*}
\frac{d-1}{2}+\frac{m-1}{d} & = \frac{d(d-1)+ 2(m-1)}{2d} \le \frac{2(m-2) + 2(m-1)}{2d} \\
& = \frac{2m-3}{d} \le \frac{2m}{3} - 1 \le  \frac{2n}{3\ell} - 1 \le \frac{n}{\ell} - \frac{7}{3}, 
\end{align*}
which implies that 
\begin{equation}  \label{eq:R3-2}
\#D_{n,d}(H) \ll_n H^{\frac{n}{\ell}-\frac{7}{3}}. 
\end{equation}

Hence, summarizing from \eqref{eq:R1}, \eqref{eq:R2}, \eqref{eq:R3-1} and \eqref{eq:R3-2}, we obtain 
\begin{equation}   \label{eq:R4}
\#D_n(H) = \#D_{n,n/\ell}(H) + O_n(H^{\frac{n}{\ell}-2}).  
\end{equation}

Finally, for a polynomial $f \in \Z[x]$ contributing to $D_{n,n/\ell}(H)$ with $f=g \circ h, \deg g = n/\ell$ and $\deg h=\ell$, 
the polynomial $h$ is certainly indecomposable (because it is of prime degree). 
Suppose that $g$ is decomposable. Then, $f$ also contributes to some $D_{n,d}(H)$ with $d < n/\ell$. 
However, in this case, according to \eqref{eq:R1}, \eqref{eq:R2}, \eqref{eq:R3-1} and \eqref{eq:R3-2},  we have 
$$
\#D_{n,d}(H) \ll_n H^{\frac{n}{\ell} - 2}. 
$$
This implies 
\begin{equation}  \label{eq:DdId}
\#D_{n,n/\ell}(H) = \#I_n(H) + O_n(H^{\frac{n}{\ell}-2}). 
\end{equation}
So, the desired result follows from \eqref{eq:R4} and \eqref{eq:DdId}.

\subsection{Proof of Theorem~\ref{thm:monicasy0}}

The only possibility to express $f \in D_4(H)$ in the form $f(x)=g(h(x))$
with monic origin-preserving $g,h \in \Z[x]$, both of degree greater than $1$, is when $g$ and $h$ are both quadratic. Set $g(x)=x^2+a_1x$ and $h(x)=x^2+bx$ with $a_1,b \in \Z$.
Then, 
$$
f(x)=(x^2+bx)^2+a_1(x^2+bx)=x^4+2bx^3+(b^2+a_1)x^2+a_1bx.
$$ 
Clearly, different pairs $(b,a_1) \in \Z^2$ produce different polynomials $f$ (see Lemma~\ref{lem:injective}), 
so it suffices to count those integer pairs that satisfy the inequalities
\begin{equation}\label{eee1}
 |b| \leq H/2, \>\> -H-b^2 \leq a_1 \leq H-b^2 \>\> \text{and} \>\> |a_1b| \leq H. 
 \end{equation}

 Fix $b \in \Z$, $|b| \leq H/2$. Set $V(b,H)$ for the number $a_1 \in \Z$ for which all the inequalities in \eqref{eee1} hold. 
It is clear that 
$V(0,H)=2H+1$
and $V(b,H)=V(-b,H)$. By \eqref{viens00}, for each $b \in \N$ satisfying \eqref{eee1} we must have $b < 6\sqrt{H}$. 
Hence, 	
\begin{equation}\label{vei1}
\#D_4(H)=2H+1+2\sum_{1 \leq b <6\sqrt{H}} V(b,H). 
\end{equation}

 Fix any integer $b$ in the range $1 \leq b<6\sqrt{H}$. Then, the last inequa\-li\-ty in \eqref{eee1} is equivalent to $-H/b \leq a_1 \leq H/b$. Since $-H-b^2<-H/b$ and $6\sqrt{H} \leq H/2$ for $H \geq 144$, the number $V(b,H)$ for $H \geq 144$ is equal to the number of integers $a_1$ satisfying
\begin{equation}\label{vei2}
-H/b \leq a_1 \leq \min\{H-b^2,H/b\}.
\end{equation}

In the range $\sqrt{H/2} < b < 6 \sqrt{H}$
there are only $O(\sqrt{H})$ suitable $b$ and for each of them at most
$O(\sqrt{H})$ suitable $a_1$ satisfying \eqref{eee1}. Hence, 
$$
\sum_{\sqrt{H/2} < b<6\sqrt{H}} V(b,H)=O(H). 
$$
Also, $V(1,H)=2H$ by \eqref{vei2}. 
Therefore, by \eqref{vei1}, we find that
$$\#D_4(H)=2\sum_{2 \leq b \leq \sqrt{H/2}} V(b,H)+O(H). 
$$

Note that for each $b$ satisfying $2 \leq b \leq \sqrt{H/2}$ we have
$H/b \leq H-b^2$ provided that $H$ is large enough. So, by \eqref{vei2},
we deduce that $V(b,H)=2\lfloor H/b\rfloor+1$ for each of those $b$.
Therefore, 
$$\sum_{2 \leq b \leq \sqrt{H/2}} V(b,H)=\sum_{2 \leq b \leq \sqrt{H/2}} (2\lfloor H/b\rfloor+1) = H\log H +O(H),$$
which implies the desired result.

\subsection{Proofs of Theorems~\ref{thm:monicasy00} and \ref{thm:monicasy}}

By \eqref{eq:Dn},  for $n=6$ the set $D_{6}(H)$ consists of the union of two sets
$D_{6,3}(H)$ and $D_{6,2}(H)$. Note that Theorem~\ref{thm:monic-nd} implies
\begin{equation}\label{boi0}
\#D_{6,2}(H)=O(H^{\frac{3}{2}}).
\end{equation}
 Later, we will consider the contribution
from $D_{6,3}(H)$. 

Similarly, by Theorem~\ref{thm:monic-nd}, for even $n \geq 8$ the main contribution comes from
$D_{n,d}(H)$ with $d=n/2 \geq 3$, while the contributions of other
$D_{n,d}(H)$ (with $2 \leq d < n/2$) are of the form 
\begin{equation}\label{boi1}
O_n(H^{\frac{n}{2}-2}), 
\end{equation}
which in fact follows from \eqref{eq:R1}, \eqref{eq:R2}, \eqref{eq:R3-1} and \eqref{eq:R3-2} 
(taking $\ell = 2$ there).

In all what follows we will evaluate $\#D_{n,d}(H)$ with $d=n/2$ for even $n \ge 6$. 
 
By Lemma~\ref{lem:injective}, each $f \in D_{n,d}(H)$ is expressed uniquely
as $f(x)=g(x^2+bx)$ for some $b \in \Z$ and some monic origin-preserving polynomial $g$ of degree $d$, so we need to evaluate 
the number of polynomials of the form
$$f(x)=(x^2+bx)^{d}+a_{d-1}(x^2+bx)^{d-1}+\dots+a_1(x^2+bx)=
x^{2d}+\sum_{j=1}^{2d-1} c_j x^j,$$
where the coefficient $c_j$ for $x^{j}$, $1 \leq j \leq 2d-1$, is 
\begin{equation}\label{coff1}
c_j=\sum_{j/2 \leq k \leq j} {k \choose 2k-j} b^{2k-j} a_k,
\end{equation}
where $a_d=1$ and $a_k =0$ for $k>d$. Fix $b \in \Z$. Let $V(b,H)$ be the number of vectors $(a_1,\dots,a_{d-1}) \in \Z^{d-1}$ for which $\rH(f) \leq H$.
This is equivalent to $|c_1|,\dots,|c_{2d-1}| \leq H$.  

For $b=0$ we have $f(x)=x^{2d}+a_{d-1}x^{2d-2}+\dots+a_1x^2$, and hence
\begin{equation}\label{coff101}
V(0,H)=(2H+1)^{d-1}.
\end{equation}
Also, by \eqref{coff1}, it is clear that 
$$V(b,H)=V(-b,H).$$ 

Assume now that $b \ge 1$.  Since $b=\rH(x^2+bx)$ and $n=2d$,
we have $b <6H^{1/d}$ by \eqref{viens00}.
Combining this with 
\eqref{coff101} we get
\begin{equation}\label{coff103}
	\#D_{n,d}(H)=(2H+1)^{d-1}+2\sum_{1 \leq b<6H^{1/d}} V(b,H).
\end{equation}

Next, we will split the sum in \eqref{coff103}
into three sums. The first sum is over the range $1 \leq b < B$, where
$B$ is a positive integer to be chosen later. The second sum is over the range
$B \leq b \leq (H/10)^{1/d}$, and the third one is over the range $(H/10)^{1/d}<b<6H^{1/d}$.

Note that, by \eqref{coff1}, the inequality $|c_1|=|ba_1| \leq H$ is equivalent to
$|a_1| \leq H/b$.  Further, from $c_2=a_1+b^2a_2$ we see that 
$|c_2| \leq H$ is equivalent to $|a_2+a_1/b^2| \leq H/b^2$. This implies 
$$|a_2| \leq H/b^2+H/b^3 \leq 2H/b^2.$$ 
Likewise,
for each $j$ in the range $1 \leq j \leq d-1$ from \eqref{coff1} we see that $|c_j| \leq H$
is equivalent to
\begin{equation}\label{coff104}
\Big|a_j+\sum_{j/2 \leq k \leq j-1} {k \choose 2k-j} b^{2k-2j} a_k\Big|\leq H/b^j.
\end{equation}
 Applying \eqref{coff104} step by step to $j=3$, etc. up to $j=d-1$, we derive the inequalities
\begin{equation}\label{coff2}
|a_j| \leq u_jH/b^j,
\end{equation}
where $j=1,\dots,d-1$ and $u_j$ depends on $j$ only. (Note that
\eqref{coff2} also follows from Lemma~\ref{lem:height}, but here we derive \eqref{coff2} possibly with worse constants using only $|c_j| \leq H$ for $j=1,\dots,d-1$.)

Similarly, for $d \leq j \leq 2d-1$ 
using \eqref{coff1} we can rewrite the inequality $|c_j| \leq H$ in the form
\begin{equation}\label{coff4}
	\Big|{d \choose 2d-j} b^{2d-j}+\sum_{j/2 \leq k \leq d-1} {k \choose 2k-j} b^{2k-j} a_k \Big| \leq H.
\end{equation}
In particular, for $j=2d-1$ we have $c_{2d-1}=bd$. For $H$ large enough this is less than
$H$ due to $b<6H^{1/d}$, so we can restrict $j$ to the range $d \leq j \leq 2d-2$.

Suppose first that $1 \leq b <B$, where $B$ is a fixed positive integer. 
Let $$v_b=v_b(d)$$ be the volume of the body $(x_1,\dots,x_{d-1}) \in \R^{d-1}$ restricted by
 \begin{equation}\label{boy23}
 \Big|x_jb^j+\sum_{j/2 \leq k \leq j-1} {k \choose 2k-j} b^{2k-j} x_k\Big|\leq 1
 \end{equation}
 for $j=1,\dots,d-1$,  and 
 \begin{equation}\label{boy24}
 \Big|\sum_{j/2 \leq k \leq d-1} {k \choose 2k-j} b^{2k-j} x_k \Big| \leq 1,
\end{equation}
 for $ j=d,\dots,2d-2$. Then, the number of integer vectors 
 $(a_1,\dots,a_{d-1})$ satisfying
 \begin{equation}\label{coff1044}
 	\Big|a_j+\sum_{j/2 \leq k \leq j-1} {k \choose 2k-j} b^{2k-2j} a_k \Big|\leq T/b^j
 \end{equation}
for $j=1,\dots,d-1$
and 
 \begin{equation}\label{coff44}
 	\Big|\sum_{j/2 \leq k \leq d-1} {k \choose 2k-j} b^{2k-j} a_k \Big| \leq T
 \end{equation}
 for $j=d,\dots,2d-2$ 
  in the corresponding blown-up body is
 $$v_b T^{d-1}+O(T^{d-2})$$ (see, e.g., \cite{peter}). 
 
 Observing that the term ${d \choose 2d-j} b^{2d-j}$ is less than $B_0:=2^d B^{2d}$ and applying this formula to $T=H-B_0$ and $T=H+B_0$ we see that $V(b,H)$ determined by \eqref{coff104} and
 \eqref{coff4} (compare to \eqref{coff1044} and \eqref{coff44})
 is at least $$v_b(H-B_0)^{d-1}+O((H-B_0)^{d-2})$$ and at most 
 $$v_b(H+B_0)^{d-1}+O((H+B_0)^{d-2}).$$ Consequently,
 \begin{equation}\label{coff106}
 	V(b,H)=v_bH^{d-1}+O(H^{d-2})
 	\end{equation}
 for $1 \leq b <B$. 
 
 So far $B$ is any fixed positive integer greater than $1$.
 Assume now that
 $B \leq b \leq (H/10)^{1/d}$.
For $c \in \R$ and $l \geq 1/2$, the length of the interval $|x+c| \leq l$ is $2l \geq 1$, 
so it contains $\lfloor 2l \rfloor + \xi_l$ integer points with some $\xi_l \in [0,1]$. In view of \eqref{coff104} we find that there are
\begin{equation}\label{koi}
(\lfloor 2H/b \rfloor +\xi_1)(\lfloor 2H/b^2 \rfloor +\xi_2)\dots(\lfloor 2H/b^{d-1} \rfloor + \xi_{d-1})
\end{equation}
integer vectors $(a_1,\dots,a_{d-1})$ determined by the inequalities
\begin{equation}\label{coff22}
|c_1|,\dots,|c_{d-1}| \leq H.
\end{equation}
 Here, $0 \leq \xi_1,\dots,\xi_{d-1} \leq 1$ and $B \leq b \leq (H/10)^{1/d}$.
 
 Now we will choose $B \in \N$.
We claim that there is a positive integer $B$ that depends on $d$ only
such that for each $b$ satisfying
$b \geq B$
 under the assumption \eqref{coff104} the inequalities 
 \begin{equation}\label{koi1}
 |c_d|,\dots,|c_{2d-1}| \leq H
 \end{equation}
automatically hold. 
   
Indeed, for $j=d$, by \eqref{coff2} (which is a consequence of \eqref{coff104}) and $B \leq b \leq (H/10)^{1/d}$, 
we see that the left hand side of \eqref{coff4} is at most
$$
b^d +\sum_{d/2 \leq k \leq d-1} 2^k b^{2k-d} |a_k| < H/10 +H\sum_{j/2 \leq k \leq d-1} 2^k b^{k-d} u_k,
$$ 
which is less than $H$ for
$B$ is large enough. For $d+1 \leq j \leq 2d-2$
the left hand side of \eqref{coff4} is less than 
$$
2^db^{d-1} +H \sum_{j/2 \leq k \leq d-1} 2^k b^{k-j} u_k <2^dH^{1-\frac{1}{d}}+H \sum_{j/2 \leq k \leq d-1} 2^k b^{k-j} u_k.
$$ 
This is less than $H$ for
$b \geq B$, where $B$ is large enough.
This proves the existence of the required 
$B$.

From \eqref{koi} and the fact that \eqref{coff22} implies \eqref{koi1} 
it follows that for each $b$ in the range $B \leq b \leq (H/10)^{1/d}$ we have
\begin{align*}
V(b,H) &=(\lfloor 2H/b \rfloor+\xi_1)(\lfloor 2H/b^2 \rfloor +\xi_2)\dots(\lfloor 2H/b^{d-1} \rfloor +\xi_{d-1}) \\&=\frac{(2H)^{d-1}}{b^{\frac{d(d-1)}{2}}}+O_n\Big(\frac{H^{d-2}}
	{b^{\frac{(d-1)(d-2)}{2}}}\Big).
\end{align*}
Summing over corresponding $b$ we obtain
\begin{align*}
\sum_{B \leq b \leq (H/10)^{1/d}} \frac{(2H)^{d-1}}{b^{\frac{d(d-1)}{2}}} =&
(2H)^{d-1} \Big(\zeta\Big(\frac{d(d-1)}{2}\Big)-\sum_{b=1}^{B-1} b^{-\frac{d(d-1)}{2}}\Big) \\&+O_n\Big(H^{\frac{d-1}{2}+\frac{1}{d}}\Big).
\end{align*}
Also,
$$
\sum_{B \leq b \leq (H/10)^{1/d}} O_n\Big(\frac{H^{d-2}}
{b^{\frac{(d-1)(d-2)}{2}}}\Big)=O_n(H^{d-2})
$$
for $d \geq 4$, while for $d=3$
we have
$$\sum_{B \leq b \leq (H/10)^{1/3}} O\Big(\frac{H}
{b}\Big)=O(H \log H). 
$$
For $d \geq 4$, we have $(d-1)/2+1/d \leq d-2$, and so 
\begin{equation}\label{coff200}
	\begin{split}
	\sum_{B \leq b \leq (H/10)^{1/d}} V(b,H) =&(2H)^{d-1} \Big(\zeta\Big(\frac{d(d-1)}{2}\Big)-\sum_{b=1}^{B-1} b^{-\frac{d(d-1)}{2}}\Big) \\&+O_n(H^{d-2}).
	\end{split}
\end{equation}
For $d=3$ we obtain
\begin{equation}\label{coff201}
	\sum_{B \leq b \leq (H/10)^{1/3}} V(b,H)=(2H)^{2} (\zeta(3)-\sum_{b=1}^{B-1} b^{-3})+O(H^{\frac{4}{3}}).
\end{equation}

Finally, suppose that $(H/10)^{1/d} < b <6H^{1/d}$.
There are less than $6H^{1/d}$ of such $b$. Also, by
\eqref{coff2},  there are at most $uH^{d-1}/b^{d(d-1)/2}$
vectors $(a_1,\dots,a_{d-1})$ for which $\rH(f) \leq H$, where $u$ is a constant that depends only on $d$. 
This gives at most 
$$
6u \frac{H^{\frac{1}{d}+d-1}}{b^{\frac{d(d-1)}{2}}} < 6u \frac{H^{\frac{1}{d}+d-1}}{(H/10)^{\frac{d-1}{2}}} = O_n\Big(H^{\frac{d-1}{2}+\frac{1}{d}}\Big)
$$
such polynomials $f$. The exponent here is $4/3$ for $d=3$ and less than $d-2$ for $d \geq 4$.  Combining this with \eqref{coff103}, \eqref{coff106}, \eqref{coff200}, \eqref{coff201} we derive that
\begin{equation}\label{coff203}
	\#D_{6,3}(H)=H^2(4+2\sum_{1 \leq b <B}v_b(3) +8(\zeta(3)-\sum_{b=1}^{B-1} b^{-3}))+O(H^{\frac{4}{3}})
\end{equation}
(here $d=3$) and 
\begin{align*}
	\#D_{n,d}(H)=&(2H)^{d-1}+2H^{d-1}\sum_{1 \leq b <B}v_b(d)  \\&+2(2H)^{d-1} \Big(\zeta\Big(\frac{d(d-1)}{2}\Big)-\sum_{b=1}^{B-1} b^{-\frac{d(d-1)}{2}}\Big)+O_n(H^{d-2}),
\end{align*}
where $d=n/2 \geq 4$. 
This combined with \eqref{boi1} proves Theorem~\ref{thm:monicasy} with
the constant
\begin{equation}  \label{eq:kappa}
\kappa_n=2^{d-1}+2\sum_{b=1}^{B-1} v_b(d) + 2^d \Big(\zeta\Big(\frac{d(d-1)}{2}\Big)-\sum_{b=1}^{B-1} b^{-\frac{d(d-1)}{2}}\Big).
\end{equation}

To derive Theorem~\ref{thm:monicasy00} from \eqref{boi0} and \eqref{coff203} we will first
show that
for $d=3$ a suitable $B$ is $B=3$. Indeed, 
$c_1=ba_1$, $c_2=a_1+a_2b^2$, $c_3=b^3+2a_2b$, $c_4=3b^2+a_2$ and $c_5=3b$. Clearly, $|c_1| \leq H$ is equivalent to $|a_1| \leq H/b$, 
and $|c_2| \leq H$ is equivalent to $-(H+a_1)/b^2\leq a_2 \leq (H-a_1)/b^2$. 
This yields
$$|a_2| \leq \frac{H+|a_1|}{b^2} \leq \frac{H}{b^2}+\frac{H}{b^3}.$$
In order to see that $B=3$ is a suitable choice we need to show that the bounds on $|a_1|$ and $|a_2|$ yield $|c_3|, |c_4| \leq H$
for $3 \leq b \leq (H/10)^{1/3}$. (It is clear that $|c_5| \leq H$ for $H$ large enough.)

 Indeed, 
$$|c_3| \leq b^3+2|a_2|b \leq \frac{H}{10}+\frac{2H}{b}+\frac{2H}{b^2}\leq H \Big(\frac{1}{10}+\frac{2}{3}+\frac{2}{9}\Big)<H.$$
Also,
$$|c_4| \leq 3\Big(\frac{H}{10}\Big)^{\frac{2}{3}}+|a_2| <H^{\frac{2}{3}}+\frac{H}{9}+\frac{H}{27}<H$$
for $H$ large enough.
So we can indeed select $B=3$.

It remains to calculate $v_1=v_1(3)$ and $v_2=v_2(3)$.  For $d=3$, by \eqref{boy23} and \eqref{boy24}, $v_1$ is the volume of the body restricted by $|x_1| \leq 1$, $|x_1+x_2| \leq 1$, $|2x_2| \leq 1$
and $|x_2| \leq 1$. It is easy to see that $$v_1=\frac{7}{4}.$$ Similarly,
by \eqref{boy23} and \eqref{boy24},  
$v_2$ is the volume of the body restricted by the inequalities
$|2x_1| \leq 1$, $|x_1+4x_2| \leq 1$, $|4x_2| \leq 1$ and $|x_2| \leq 1$.  By a simple calculation, we find that $$v_2=\frac{7}{16}.$$ By \eqref{coff203} with $B=3$, it follows that
\begin{align*}
\#D_{6,3}(H)&=H^2\Big(4+\frac{7}{2}+\frac{7}{8}+8\Big(\zeta(3)-1-\frac{1}{8}\Big) \Big)+O(H^{\frac{4}{3}})
\\&=H^2\Big(8\zeta(3)-\frac{5}{8}\Big) +O(H^{\frac{4}{3}}). 
\end{align*}
This completes the proof of Theorem~\ref{thm:monicasy00}
due to \eqref{eq:Dn} and \eqref{boi0}.

\section{Proofs for the non-monic case} 
\label{sec:non-monic}

\subsection{Proof of Theorem \ref{thm:non-monic-nd}}

The strategy is the same as that in the proof of Theorem \ref{thm:monic-nd}. 

Write $n=dm$. 
Note that for any polynomial $f \in D_{n,d}^*(H)$, 
we have $f=g\circ h$ for some origin-preserving polynomials $g$ and $h$ with $\deg g=d$ and $\deg h=m$.  
By Lemma~\ref{lem:height},  we obtain 
$$
\rH(h) \leq |a_d| \rH(h) \le c_1 \rH(f)^{\frac{1}{d}} \le c_1 H^{\frac{1}{d}}, 
$$
where $c_1 > 0$ is a constant depending only on $n$. 
Besides, by Lemma~\ref{lem:height}, there exists a constant $c_2 > 0$ depending only on $n$ such that  
the absolute value of the coefficient of each monomial $x^j$ in $g$, $1 \le j \le d$, is at most $c_2 \rH(f)/\rH(h)^{j}$. 
Hence, as before, by estimating the possibilities of $g$ and $h$, we deduce that 
\begin{equation*}   
\begin{split}
\#D_{n,d}^*(H) & \le \sum_{1 \leq b \leq c_1 H^{1/d}} 2(2b+1)^{m-1} \cdot \frac{(c_2H)^{d}}{b^{1+2+\cdots + d}} \\
& \ll_n  \sum_{1 \leq b \leq c_1 H^{1/d}}b^{m-1}\cdot \frac{H^{d}}{b^{1+2+\cdots + d}} \\
& \ll_n H^{d}\sum_{1 \leq b \leq c_1 H^{1/d}}\frac{1}{b^{\frac{d(d+1)}{2}+1-m}}  \\
& \ll_n  \left\{
\begin{array}{rl}
H^d & \textrm{if $d(d+1) > 2m$}, \\
H^d \log H  & \text{if $d(d+1) = 2m$}, \\ 
 H^{\frac{d-1}{2}+\frac{m}{d}} & \text{if $d(d+1) < 2m$}. 
\end{array} \right.
 \end{split}
\end{equation*} 
The desired upper bounds then follow. 

For the lower bound in \eqref{gui11} note that
by counting polynomials of the form 
\begin{align*}
(a_d x^{d} + & a_{d-1} x^{d-1} + \cdots + a_1 x ) \circ x^m  \\
& = a_d x^{dm} + a_{d-1} x^{(d-1)m} + \cdots + a_1 x^m, 
\end{align*}
where  $1 \leq a_i \leq H$ for each $1\le i \le d$ and $\gcd(a_d, \dots, a_1)=1$, by Lemma~\ref{lem:factor11}, we get 
$$
 D_{n,d}^*(H) \ge 0.355H^d. 
$$ 
This completes the proof of \eqref{gui11}.

To prove the lower bound in \eqref{gui31}
we consider origin-preserving $h \in \Z[x]$ of degree $m$ with positive relatively prime coefficients satisfying 
$\rH(h) \leq H^{1/d}/(2m)$ 
and origin-preserving $g(x)=a_d x^d+a_{d-1}x^{d-1}+\dots+a_1x \in \Z[x]$ with relatively prime coefficients satisfying 
$1 \leq a_i  \leq H^{(d-i)/d}$ for $i=1,\dots,d$. 
Then, as $\rH(f) \le \rL(f) = \rL(g \circ h)$  and $\rL(h) \le H^{1/d}/2$,  using \eqref{eq:length} we obtain
$$
\rH(f) \leq \sum_{i=1}^{d} |a_i| \rL(h)^i  \leq H \sum_{i=1}^{d} \frac{1}{2^i} <H.
$$
By Lemma~\ref{lem:injective-Z}, distinct pairs $(g,h)$ give distinct $f$ with relatively prime coefficients satisfying $\rH(f) < H$. 
By Lemma~\ref{lem:factor11}, for $H$ large enough there are at least  
\begin{align*}
0.35 \big(\frac{H^{\frac{1}{d}}}{2m} \big)^m \cdot 0.35 \prod_{i=1}^{d} H^{\frac{d-i}{d}} > \frac{0.12H^{\frac{d-1}{2}+\frac{m}{d}}}{(2m)^m} 
\end{align*}
 of such pairs. 
This proves the lower bound in \eqref{gui31} for $H$ large enough.

The proof of the lower bound in \eqref{gui21} is similar to that of \eqref{gui2}. As above, we consider origin-preserving $h$
with, say, leading coefficient exactly $B$ 
 and other $m-1$ coefficients all being in $\{1,\dots,B\}$ and coprime. By Lemma~\ref{lem:factor11}, there are at least $0.355B^{m-1}$ of such $h$. 
(Note that $(d+1)d=2m$ combined with $d \geq 2$ yields $m \geq 3$, so $m-1 \geq 2$.)
To each of them we assign the origin-preserving polynomial $g(x)=a_dx^d+a_{d-1}x^{d-1}+\dots+a_1x \in \Z[x]$ with $1 \le a_i \leq H/(n^dB^i)$. 
Then, using Lemmas~\ref{lem:factor11} and \ref{lem:injective-Z},
and summing over $B$ as in the range 
\eqref{gui4} we complete the proof of the lower bound in \eqref{gui21}.

\subsection{Proof of Theorem \ref{thm:non-monic-R}}

Using Theorem \ref{thm:non-monic-nd} (instead of Theorem \ref{thm:monic-nd}), 
we can get a proof by applying the same strategy as in the proof of Theorem \ref{thm:monic-R}. 

Here we only explain the case  when $n \ne 6$ and $n \ne \ell^2$. 
In this case, $n / \ell \ge 4$. 

Now, as before, using Theorem~\ref{thm:non-monic-nd} and Lemma~\ref{lem:factor}, 
we get that if $d(d+1) > 2m$ and $d \ne n/\ell$, then we must have 
$$
\#D_{n,d}^*(H) \ll_n H^{\frac{n}{\ell}-2}; 
$$
while if $d(d+1) = 2m$ and $d \ne n/\ell$, we must have 
$$
\#D_{n,d}^*(H) \ll_n H^{\frac{n}{\ell}-2} \log H. 
$$
Also, similarly as before, using Theorem~\ref{thm:non-monic-nd}, if $d(d+1) < 2m$ and $d=2$, we obtain
$$
\#D_{n,d}^*(H) \ll_n  H^{\frac{n}{\ell}-\frac{3}{2}}; 
$$
while if $d(d+1) < 2m$ and $d \ge 3$, we must have 
$$
\#D_{n,d}^*(H) \ll_n  H^{\frac{n}{\ell}-\frac{8}{3}}. 
$$
Then, the desired result follows.

\section*{Acknowledgements} 
The authors would like to thank the referee for his/her careful reading and  valuable comments. 
The work was motivated by a talk of Professor Joachim von zur Gathen at the University of New South Wales. 
The authors would like to thank him for his stimulating talk and for his valuable comments. 
The research of the first named author has received funding from
European Social Fund (project No 09.3.3-LMT-K-712-01-0037) under grant agreement with
the Research Council of Lithuania (LMTLT).
The research of the second named author was partly supported by the Guangdong Basic and Applied Basic Research Foundation (No. 2022A1515012032), 
by a Macquarie University Research Fellowship and by the Australian Research Council Grant DE190100888.

\end{document}